\documentclass[11pt]{amsart}

\usepackage[centertags]{amsmath}
\usepackage{amsfonts}
\usepackage{amssymb}
\usepackage{amsthm}
\usepackage[english]{babel}
\usepackage[active]{srcltx}

\usepackage[colorlinks]{hyperref}

\newtheorem{thm}{Theorem}[section]
\newtheorem{cor}[thm]{Corollary}
\newtheorem{lem}[thm]{Lemma}
\newtheorem{prop}[thm]{Proposition}

\newtheorem{ex}[thm]{Example}

\newtheorem{defn}[thm]{Definition}
\newtheorem{notation}[thm]{Notation}
\newcommand{\R}{\mathbb{R}}
\newcommand{\Real}{\mathbb{R}}

\newcommand{\B}{\mathbf{B}}

\begin{document}

\title[Global Inversion of Nonsmooth Mappings]{Global Inversion of Nonsmooth Mappings Using Pseudo-Jacobian Matrices}

\author[Jes\'us A. Jaramillo, \'Oscar  Madiedo and Luis S\'anchez-Gonz\'alez]{Jes\'us A. Jaramillo$^\dag$, \'Oscar  Madiedo$^\dag$ and Luis S\'anchez-Gonz\'alez$^\ddag$}

\address{$^\dag$ Departamento de An{\'a}lisis Matem{\'a}tico\\ Facultad de
Matem{\'a}ticas\\ Universidad Complutense de Madrid\\ 28040 Madrid, Spain}

\email{jaramil@mat.ucm.es, oscar.reynaldo@mat.ucm.es}

\address{$^\ddag$ Departamento de Ingenier{\'i}a Matem{\'a}tica\\ Facultad de CC. F{\'i}sicas y  Matem{\'a}ticas\\ Universidad de Concepci{\'o}n\\ Casilla 160-C, Concepci{\'o}n, Chile}

\email{lsanchez@ing-mat.udec.cl}

\thanks{Research supported in part by MICINN Project MTM2012-34341 (Spain). O. Madiedo is also supported by grant MICINN {BES2010-031192}.  L. S\'anchez-Gonz\'alez has also been  partially  supported by the FONDECYT project 11130354 (Chile).}



\date{November 2013}                                           

\maketitle

\begin{abstract}
We study the global inversion of a continuous nonsmooth mapping   $f:{\mathbb R^n} \to {\mathbb R^n}$, which may be non-locally Lipschitz.  To this end, we use the notion of pseudo-Jacobian map associated to  $f$, introduced by Jeyakumar and Luc, and we consider a related index of regularity for $f$. We obtain a characterization of global inversion in terms of  its index of regularity. Furthermore, we prove that the Hadamard integral condition has a natural counterpart in this setting,  providing a sufficient condition for global invertibility.

\end{abstract}

\section{Introduction}
Global inversion of mappings is a relevant issue in nonlinear analysis. In the case that  $f:{\mathbb R^n} \to {\mathbb R^n}$ is a $C^1$-smooth  mapping with everywhere nonzero Jacobian, the problem of global invertibility of $f$  was first considered by Hadamard \cite{Hadamard},  who obtained a sufficient condition in terms of the growth of the norm of the inverse of the derivative  $df(x)$, by means of his celebrated  {\it integral condition}. Namely, he proved that $f$ is  a global diffeomorphism provided
$$
\int _0^\infty \inf_{\vert\vert x \vert\vert = t}\left\|df(x)^{-1}\right\|^{-1} \, dt=\infty.
$$

This result has found a wide number of extensions and variants  in different contexts. In this way, an extension of Hadamard integral condition to the case of local diffeomorphisms between Banach spaces was given by Plastock in \cite{Plastock}.  The global invertibility of local diffeomorphisms between Banach-Finsler manifolds was studied by  Rabier in \cite{Rabier}. In \cite{John},  John obtained a variant of  Hadamard integral condition for a local homeomorphism $f$ between Banach spaces, using  the lower scalar Dini derivative of $f$. More recently, the problem of global inversion of a local homeomorphism between metric spaces has been considered in \cite{GutuJaramillo} and \cite{GaGuJa}, where some versions of Hadamard integral condition are obtained in terms of an analogue of lower scalar Dini derivative.

In the nonsmooth setting,  Pourciau studied in \cite{Pourciau1982} and \cite{Pourciau1988} the global inversion of locally Lipschitz mappings $f: \mathbb R^n \to \mathbb R^n$ using the Clarke generalized differential, and he also obtained a variant of Hadamard integral condition in this context. The global invertibility of locally Lipschitz mappings between (finite-dimensional) Finsler manifolds, using an analogue of the Clarke generalized differential, has been studied in \cite{JaMaSa}.

\smallskip

Our main purpose in this paper is to use techniques of nonsmooth analysis in order to study the global inversion of a continuous mapping $f: \mathbb R^n \to \mathbb R^n$, which may be non-locally Lipschitz, and for which the Clarke generalized differential is not necessarily  defined. To this end, we will use the concept of {\it pseudo-Jacobian} (also called {\it approximate Jacobian}) of the mapping $f$, introduced by Jeyakumar and Luc in \cite{JL0} and  studied later in \cite{JL1, Luc, JL}. In particular, we will see that the Hadamard integral condition has a natural counterpart in this setting, and provides a sufficient condition for global invertibility.

\smallskip

The paper is organized as follows. In Section 2 we include some basic definitions and preliminary results  which will be useful throughout the paper. In Section 3, for a mapping $f: \mathbb R^n \to \mathbb R^n$ we introduce the regularity index of $f$ related to a pseudo-Jacobian mapping $Jf$,  and we see its connection with  the lower scalar  Dini derivative of $f$. In Section 4 we prove our main results,  Theorem \ref{TFIG} gives a characterization of global inversion of a mapping $f: \mathbb R^n \to \mathbb R^n$ in terms of the regularity index of  $f$. In Corollary \ref{estimate}, we obtain an integral estimation of the domain of invertibility of $f$ around a point. Finally,  Corollary \ref{Hada} gives a version of Hadamard integral condition using the regularity  index, which provides a sufficient condition for global inversion.

\section{Preliminaries}

Let us begin by recalling the definition of pseudo-Jacobian associated to a continuous mapping, as well as some basic properties. This notion was introduced in \cite{JL0}, where it was called approximate Jacobian. We refer to the book \cite{JL} for an extensive information about this concept and its applications. The notation we use is standard. The $n$-dimensional Euclidean space is denoted by $\R^n$, and $\mathcal{L}(\Real^n,\Real^m)$ denotes the space of all linear mappings from $\Real^n$ to $\Real^m$, which can be regarded as  $m\times n$-matrices. The space $\mathcal{L}(\Real^n,\Real^m)$ is endowed with its usual matrix norm. The open unit ball of $\R^n$ and $\mathcal{L}(\Real^n,\Real^m)$ are denoted, respectively, by $\B_n$ and $\B_{m\times n}$.

\begin{defn}\label{pseudo}
Let $f: \R^n \to \R^m$ be a continuous mapping. We say that a nonempty closed
set of $m\times n$ matrices $Jf(x)\subset \mathcal{L}(\R^n, \R^m)$ is a pseudo-Jacobian
of $f$ at $x$ if for every $u \in \R^n$ and $v\in \R^m$ one has
\begin{center}
$(vf)^+(x;u) \leq \sup\limits_{M \in Jf(x)}\langle v, M u \rangle$,
\end{center}
where $vf$ is the real function $(vf)(x) = \sum\limits_{i=1}^{m}v_i f_i (x)$ for every
$x\in \R^n$, ($v_i$ being components of $v$ and $f_i$ being components of $f$),
and $(vf)^{+}(x;u)$ is the upper Dini directional derivative of the function $vf$ at
$x$ in the direction $u$, that is
$$
(vf)^{+}(x;u) :=\limsup\limits_{t  \to 0^{+}}\frac{(vf)(x + tu)-(vf)(x)}{t}.
$$
Each element of $Jf(x)$ is called a {pseudo-Jacobian matrix} of $f$ at $x$. If for every $x\in \mathbb R^n$ we have that $Jf(x)$ is a pseudo-Jacobian of $f$ at $x$, we say that  the set-valued map $Jf: \R^n \rightrightarrows \mathcal{L}(\R^n, \R^m)$ given by $Jf:x\mapsto Jf(x)$ is  a pseudo-Jacobian map for $f$.
\end{defn}

If $f: \R^n \to \R^m$ is continuous and G\^{a}teaux-differentiable at $x$ with derivative $df(x)\in \mathcal{L}(\R^n, \R^m)$, then, of course, $Jf(x):=\{df(x)\}$ is a pseudo-Jacobian of $f$ at $x$. Moreover, as can be seen in \cite[Section 1.3]{JL}, many generalized derivatives frequently used in nonsmooth analysis are examples of pseudo-Jacobians. In particular, this is the case of Clarke generalized Jacobian of a locally Lipschitz mapping. Nevertheless, there are  examples  of  locally Lipschitz functions whose Clarke generalized Jacobian strictly contains a pseudo-Jacobian. Consider, for example,  the following illustrative  case taken from \cite{JL}. A pseudo-Jacobian of the function  $f:\R^2\to\R^2$ defined by $f(x,y) = (|x|,|y|)$ at $(0, 0)$ is the set
$$
Jf(0, 0) = \left\{ \left(
											   \begin{matrix} 
											      1 & 0 \\
											      0 & 1 \\
											   \end{matrix}\right),
											  \left(
											   \begin{matrix} 
											      1 & 0 \\
											      0 & -1 \\
											   \end{matrix}\right),
											   \left(
											   \begin{matrix} 
											      -1 & 0 \\
											      0 & 1 \\
											   \end{matrix}\right),
											  \left(
											   \begin{matrix} 
											      -1 & 0 \\
											      0 & -1 \\
											   \end{matrix}\right)\right\}.
$$
Whereas the {Clarke generalized
Jacobian} is given by
$$
\partial^{C}f(0, 0) = \left\{\left(
												\begin{matrix} 
											      \alpha & 0 \\
											      0 & \beta \\
											   \end{matrix}
											\right): \alpha,\beta \in [-1,1]\right\},
$$
which is also a pseudo-Jacobian of $f$ at $(0, 0)$ and contains $Jf(0,0)$, in fact,  $\partial^{C}f(0,0)$ is the convex hull of $Jf(0,0)$. However, this is not always the case (see \cite[Example 1.3.3]{JL}).

\smallskip

We will  need the following Mean Value Theorem concerning  pseudo-Jacobians,  whose proof can be seen in   \cite[Theorem 2.2.2]{JL}.

\begin{thm} \label{properties:pseudo} Let $f:\R^n \to \R^m$ be a continuous mapping and  $u, v \in \R^n$.  Suppose that $Jf(x)$ is a pseudo-Jacobian of $f$ at $x$ for each $x$ in the segment $[u,v]$. Then
$$f(u) - f(v) \in \overline{co}(Jf([u, v])(u-v)).$$

\end{thm}

\smallskip

Throughout the paper,  we will be interested in  pseudo-Jacobians with nice stability properties, in the sense that they are upper semicontinuous.

\begin{defn}
Let $F: \R^n \rightrightarrows \R^m$ be a set-valued map. We say that F is upper
semicontinuous (usc) at x if for every $\varepsilon > 0$, there exists some $\delta > 0$
such that
$$F(x + \delta \B_{n}) \subseteq F(x) + \varepsilon \B_{m}.$$
\end{defn}

It is well-known that the Clarke generalized Jacobian of a locally Lipschitz  mapping $f: \R^n \to \R^m$ is  usc (see \cite[Proposition 2.6.2]{Clarkebook}). Recall that a set-valued map $F: \R^n \rightrightarrows \mathcal{L}(\R^n, \R^m)$ is said to be $\emph{locally bounded}$ at $x$ if there exist a neighborhood $U$ of $x$ and a constant $\alpha>0$ such that $\|A\| \leq \alpha$ for each $A \in F(U)$. Clearly, if $F$ is usc at $x$ and  $F(x)$ is bounded,
then $F$ is locally bounded at $x$. On the other hand,  it is  proved in \cite[Proposition 2.2.8]{JL} that a continuous mapping $f: \R^n \to \R^m$  has a locally bounded pseudo-Jacobian map at $x$ if,  and only if, $f$ is locally Lipschitz at $x$.

\smallskip

Finally, let us recall the definition of the lower and upper scalar Dini
derivatives in the setting of continuous mappings between Banach spaces. These quantities
were used by John \cite{John} to study local homeomorphisms between Banach spaces,
in order to obtain a version of Hadamard integral condition. Later on, this kind of
scalar derivatives have been considered in \cite{GutuJaramillo} and \cite{GaGuJa}
in a metric space setting.

\begin{defn}
Let $E$ and $F$ be Banach spaces,  $U\subset E$ an  open set and $f:U \to F$  a continuous mapping.  The lower and upper scalar derivatives of $f$ at a point
$x\in U$  are defined respectively as
\begin{equation*}
D^{-}_x f =\liminf_{y\to x} \frac{\|f(y) - f(x)\|}{\|y - x\|}, \quad D^{+}_x f =\limsup_{y\to x} \frac{\|f(y) - f(x)\|}{\|y - x\|}
\end{equation*}
where $y\in U$ and $y\neq x$.
\end{defn}

We will also need the following Mean Value inequality given in  \cite[Proposition 3.9]{GutuJaramillo}. Firstly, recall that if $p: [a,b]\to E$ is a continuous path in a Banach space $E$, the {\it length} of $p$ is defined by
$$
\ell(p) = \sup \sum_{i=0}^{n-1}\| p(t_{i+1})- p(t_i)\|,
$$
where the supremum is taking over all partitions $a=t_0\leq t_1\leq\cdots\leq t_n=b$. The path $p$ is said to be {\it rectifiable} if $\ell(p)<\infty$.

\begin{prop}\label{MVI}
Let $E$ and $F$ be Banach spaces,   $U\subset E$ an open set and  $f:U \to F$  a continuous mapping. Suppose that  $q: [a,b]\to U$ is a continuous
path such that $p = f\circ q: [a,b]\to F$ is rectifiable and $q(a)\neq q(b)$. Then there exists $\tau \in [a,b]$ such that
$$
\ell (p) \geq D_{q(\tau)}^- f\cdot \|q(b)-q(a)\|.
$$
\end{prop}


\section{Regularity index and lower scalar derivative}

In this section we study the equi-invertibility of pseudo-Jacobian matrices,  and how it is related to the lower scalar derivative. Let us recall that, according to  Pourciau \cite{Pourciau1988}, the co-norm of a matrix  $A\in \mathcal{L}(\Real^n,\Real^n)$ is defined as
\begin{equation*}
/\hspace{-0.1cm}/A/\hspace{-0.1cm}/ = \inf_{||u||=1}||Au||.
\end{equation*}
It is clear that the matrix $A$ is invertible if, and only if, $/\hspace{-0.1cm}/A/\hspace{-0.1cm}/>0$. A subset of matrices $\mathcal A \subset \mathcal{L}(\Real^n,\Real^n)$ is said to be  \emph{equi-invertible} provided
$$\inf \{/\hspace{-0.1cm}/A/\hspace{-0.1cm}/ \, : \, A \in \mathcal A \}>0.$$

\smallskip

The concept of regularity for pseudo-Jacobian maps is defined as follows.

\begin{defn}
Let $f: \R^n \to \R^n$ be a continuous mapping, and let $Jf: \R^n \rightrightarrows \mathcal{L}(\R^n, \R^n)$ be a pseudo-Jacobian map of $f$. The associated regularity index $\alpha_{Jf}(x)$ of $f$ at $x\in\R^n$ is defined by
$$
\alpha_{Jf}(x):=
\inf\{/\hspace{-0.9mm}/A/\hspace{-0.9mm}/: A \in co(Jf(x))\}.
$$
We say that $f$ is $Jf$-regular at $x$ if   $\alpha_{Jf}(x) >0$.  When $f$ is $Jf$-regular at $x$ for every $x\in \Real^n$, we say that $f$ is $Jf$-regular.
\end{defn}

The connection of the regularity index with the original Hadamard integral condition can be seen as follows. Let $f: \R^n \to \R^n$ be a a $C^1$-smooth  mapping with everywhere nonzero Jacobian and consider the natural pseudo-Jacobian of $Jf$ of $f$ given by $Jf(x):=\{df(x)\}$. Then it is easy to see that
$$
\alpha_{Jf}(x) = \left\|df(x)^{-1}\right\|^{-1}.
$$

\smallskip

In order to analyze the behavior of the regularity index, we shall use the following notation:

\begin{notation}
{\rm Let $f: \R^n \to \R^n$ be a continuous mapping and let  $Jf: \R^n \rightrightarrows \mathcal{L}(\R^n, \R^n)$ be a pseudo-Jacobian map of  $f$. For each $\beta>0$, we denote:}
$$
\alpha_{Jf}(x,\beta):= \inf\{/\hspace{-0.9mm}/A/
\hspace{-0.9mm}/: A \in co(Jf(x + \beta {\B}_n))\}.
$$
\end{notation}

We will need the following simple, but useful, lemmas.

\begin{lem}\label{conbeta}
Let $f: \R^n \to \R^n$ be a continuous mapping and $Jf$  a pseudo-Jacobian map of $f$. If $Jf$ is usc at a point $x\in \R^n$, then
$$
\lim\limits_{\beta\to 0^{+}}\alpha_{Jf}(x,\beta) =
\sup\limits_{\beta>0}\alpha_{Jf}(x,\beta) = \alpha_{Jf}(x).
$$
\end{lem}
\begin{proof}
First of all, notice that $\alpha_{Jf}(x) \geq \alpha_{Jf}(x,\beta_1) \geq \alpha_{Jf}(x,\beta_2)$ whenever  $0<\beta_1 <\beta_2$, so the first equality is clear. Now, since $Jf$ is usc at $x$,  given $\varepsilon >0$ there is some $\beta >0$ such that
$$
Jf(x +\beta {\B}_n) \subset Jf(x) +\varepsilon  {\B}_{n\times n}.
$$
Then,
$$
co(Jf(x +\beta {\B}_n)) \subset co(Jf(x) +\varepsilon {\B}_{n\times n}) \subset co(Jf(x)) + \varepsilon {\B}_{n\times n}.
$$
Therefore, for each $A\in co(Jf(x+\beta {\B}_n))$ there exist $\tilde{A}\in co(Jf(x))$ and $\tilde{B} \in {\B}_{n\times n}$ such that $A=\tilde{A}+ \varepsilon \tilde{B}$. As a consequence, for every $u\in \mathbb R^n$ with $\|u\| = 1$, we have that
$$
\|Au\| \geq \|\tilde{A}u\| - \varepsilon
\|u\|\geq \alpha_{Jf}(x) -\varepsilon, $$
and hence   $/\hspace{-0.9mm}/A/\hspace{-0.9mm}/\geq \alpha_{Jf}(x)-\varepsilon$. Then, it follows  that
$$
\alpha_{Jf}(x) \ge \alpha_{Jf}(x,\beta)  \ge \alpha_{Jf}(x)-\varepsilon,
$$
for all $\varepsilon>0$, and the second equality holds.
\end{proof}

\smallskip

\begin{lem}\label{beta}
Let $f: \R^n \to \R^n$ be a continuous mapping and let $Jf$ be a pseudo-Jacobian
map of $f$. If there are $x\in \Real^n$ and $\beta>0$ such that $\alpha_{Jf}(x,\beta)>0$, then
\begin{equation*}
\|f(x + h) - f(x)\| \geq  \alpha_{Jf}(x,\beta)\cdot \|h\| \quad \text{for all} \quad  \|h\|< \beta. \label{open function1}
\end{equation*}

\end{lem}
\begin{proof}
Choose $0<\|h\|< \beta$. By the Mean Value Theorem (Theorem \ref{properties:pseudo})  we have that, for any fixed $0<\varepsilon \leq \alpha_{Jf}(x,\beta)$,
$$
f(x + h) - f(x) \in \overline{co}(Jf([x,x+h])(h)) \subseteq \overline{co}(Jf(x + \beta {\B}_n)(h))
$$
$$
\subseteq co (Jf(x + \beta {\B}_n)(h)) + \varepsilon \|h\| {\B}_{n} = co ( Jf(x + \beta {\B}_n))(h) + \varepsilon \|h\| {\B}_{n}.
$$
Thus, there exist $A\in co(Jf(x+\beta {\B}_n))$ and  $v \in {\B}_{n}$ such that $f(x + h) - f(x) = Ah +
\varepsilon \|h\| v$.  Then
\begin{equation*}
\|f(x + h) - f(x)\| \geq \|Ah\| - \varepsilon \|h\|  \geq
(\alpha_{Jf}(x,\beta) - \varepsilon) \|h\|.
\end{equation*}
As a consequence, we obtain that $\|f(x + h) - f(x)\| \geq \alpha_{Jf}(x,\beta)\|h\|$.

\end{proof}

In the next result, which is a direct consequence of Lemmas  \ref{conbeta} and \ref{beta} above, we compare the lower scalar derivative of a continuous mapping $f$ to the regularity index of a pseudo-Jacobian of $f$. This relationship  will be one of the keys to obtain our global inversion results.

\begin{thm}\label{DerEsc1}
Let $f: \R^n \to \R^n$ be a continuous mapping and let $Jf$ be a pseudo-Jacobian map of $f$. If  $Jf$ is usc at $x \in \R^n$ and $f$ is $Jf$-regular at $x$, then
\begin{equation*}
D_{x}^{-}f \geq \alpha_{Jf}(x).
\end{equation*}
\end{thm}

\smallskip

As we have mentioned before, in the case of a continuous but  non-locally Lipschitz  mapping $f: \R^n \to \R^n$, we have that the pseudo-Jacobian map can be, in general, unbounded. In order to deal with this problem  a useful tool is the so-called the recession cone, which is defined as follows. Consider a subset  $\mathcal A \subset \mathcal{L}(\R^n,\R^n)$;  the $\emph{recession cone}$ (or \emph{asymptotic cone}) of  $\mathcal A$  is defined by
$$
\mathcal A_{\infty}:= \{\lim_{j\to \infty} t_j A_j \, :\, A_j \in \mathcal A, t_j \downarrow 0\}.
$$
We refer to \cite{JL} for further information about recession cones. Suppose now that $f: \R^n \to \R^n$ is continuous and that $Jf(x)$ is a pseudo-Jacobian of $f$ at $x$. The set $Jf(x)_{\infty}$ denotes the recession cone of
$Jf(x)$. The elements of $Jf(x)_{\infty}$ are called $\emph{recession matrices}$ of $Jf(x)$. Recall that from \cite[Proposition 3.1.6]{JL} we have that $f$ is $Jf$-regular whenever the set $\overline{co} (Jf(x)) \cup co(Jf(x)_{\infty} \backslash \{0\})$ is invertible. Thus, we can obtain the following direct consequence.

\begin{cor}\label{DerEsc}
Let $f: \R^n \to \R^n$ be a continuous mapping and let $Jf$ be a pseudo-Jacobian map of $f$. If $Jf$ is usc at $x \in \R^n$ and  each element of the set $\overline{co} (Jf(x)) \cup co(Jf(x)_{\infty} \backslash \{0\})$ is invertible, then
\begin{equation*}
D_{x}^{-}f \geq \alpha_{Jf}(x)>0.
\end{equation*}
\end{cor}

\section{Global Inversion}

First of all, we consider the problem of local inversion. Let $f : \R^n\to\R^n$ be a continuous mapping and let $x_0 \in \R^n$ be given. We say that $f$ admits locally an inverse at $x_0$ if there exist neighborhoods $U$ of $x_0$ and $V$ of $f(x_0)$, and a continuous mapping $g : V \to \R^n$ such that $g(f(x)) = x$ and $f(g(y)) = y$ for every $x \in U$ and $y \in V$. Taking into account Lemmas \ref{conbeta} and \ref{beta},  the following result follows from  \cite[Corollary 5.2]{Luc}:

\begin{thm}\label{TFI}
Let $f: \R^n\to\R^n$ be a continuous mapping and let $Jf$ be a pseudo-Jacobian map of $f$. If $Jf$ is  usc at $x_0\in \R^n$ and $f$
is $Jf$-regular at $x_0$, then $f$ admits locally an inverse at $x_0$, which is  Lipschitz  at $f(x_0)$.
\end{thm}

\smallskip


Now, let us give a characterization of global inversion of a  continuous mapping,  in terms of the regularity index of a pseudo-Jacobian map.

\begin{thm}\label{TFIG}
Let $f: \R^n\to\R^n$ be a continuous mapping and let $Jf$ be a pseudo-Jacobian map of $f$. Suppose that $Jf$ is  usc  and $f$ is $Jf$-regular on $\mathbb R^n$. The following conditions are equivalent:
\begin{itemize}
\item[(a)] $f: \R^n\to\R^n$ is a global homeomorphism.
\item[(b)] For each compact subset $K \subset \mathbb R^n$ there exists $\alpha_K>0$ such that $\alpha_{Jf}(x) > \alpha_K$ for every $x\in f^{-1}(K)$.
\end{itemize}
\end{thm}

\begin{proof}
$(a)\Rightarrow (b)$ Suppose that $f: \R^n\to\R^n$ is a homeomorphism, and let $K\subset \mathbb R^n$ be a compact set. Then  $ f^{-1}(K)$ is also compact. Now consider
$$
\alpha_K = \inf \{\alpha_{Jf}(x) \, : \, x \in f^{-1}(K)\}.
$$
We claim that $\alpha_K>0$. Indeed, otherwise we can find a sequence $(x_j)$ in $f^{-1}(K)$ such that $\alpha_{Jf}(x_j)$ converges to zero. By compactness, we can assume that $(x_j)$ is convergent to some point $x$, and we know that $\alpha_{Jf}(x)>0$. By Lemma \ref{conbeta}, there is some $\beta>0$ such that $\alpha_{Jf}(x, \beta)> \frac{1}{2}\alpha_{Jf}(x)>0$. But $x_j$ belongs to the ball $x+\beta {\B}_n$ for $j$ large enough, which  is a contradiction.

$(b)\Rightarrow (a)$ By Theorem \ref{TFI} we know that $f$ is a local homeomorphism. Thus, according to  \cite[Theorem 1.2]{Plastock}, it is sufficient to prove that $f$ satisfies the following limiting condition (L):

\medskip

\noindent {\rm (L)} For every line segment $p:[0, 1] \to \mathbb R^n$,  given by $p(t)=(1-t)y_0+ty_1$ for some $y_0, y_1\in \mathbb R^n$,  for every
$0<b\leq 1$ and for every continuous path $q:[0, b)\to \mathbb R^n$ satisfying that $f(q(t))=p(t)$ for every $t\in [0, b)$, there exists a sequence $(t_j)$ in $[0, b)$ convergent to $b$ and such that the sequence $\{q(t_j)\}$ is convergent in $\mathbb R^n$.

\medskip

In order to verify condition (L), consider a line segment $p:[0, 1] \to \mathbb R^n$ given by $p(t)=(1-t)y_0+ty_1$ for some $y_0, y_1\in \mathbb R^n$, consider some $0<b\leq 1$, and let $q:[0, b)\to \mathbb R^n$ be a continuous path satisfying that $f(q(t))=p(t)$ for every $t\in [0, b)$. The set $K:= p([0, 1])$ is compact in $\mathbb R^n$ and $q([0, b))\subset f^{-1}(K)$, so, by hypothesis, there exists $\alpha_K>0$ such that $\alpha_{Jf}(x) > \alpha_K>0$  for every $x\in q([0, b))$. Using Theorem \ref{DerEsc1} we deduce that
$$
D^-_{x} f \geq \alpha_{Jf}(x) > \alpha_K>0,
$$
for every $x\in q([0, b))$. Now by   Proposition \ref{MVI} we have that, for every $s, t \in [0, b)$ with $s<t$:
$$
\ell(p_{|[s,t]}) \geq \inf \{D^-_{x} f \, : \ x\in q([s,t])\}\cdot \Vert q(s) - q(t) \Vert.
$$
Since   $\ell(p_{|[s,t]})= \vert s-t\vert \cdot \Vert y_0-y_1\Vert$,  we obtain that
$$
\vert s-t\vert \cdot \Vert y_0-y_1\Vert \geq \alpha_K \Vert q(s) - q(t) \Vert.
$$
Now consider any sequence $(t_j)$ in $[0, b)$ convergent to $b$. The above inequality gives that
$$
\Vert q(t_i) - q(t_j) \Vert \leq \frac{1}{\alpha_K} \vert t_i - t_j\vert \cdot \Vert y_0-y_1\Vert,
$$
for every $i, j$. This shows that the sequence $\{q(t_j)\}$ is Cauchy, and therefore convergent in $\mathbb R^n$.
\end{proof}

\smallskip

As a consequence of the previous result, we obtain at once the following result:

\begin{cor}\label{bound}
Let $f: \R^n\to\R^n$ be a continuous mapping and let $Jf$ be a pseudo-Jacobian map of $f$. Suppose that $Jf$ is  usc  and there exists some $\alpha>0$ such that $\alpha_{Jf}(x) \geq \alpha$ for every $x \in \mathbb R^n$. Then  $f: \R^n\to\R^n$ is a global homeomorphism.
\end{cor}

\smallskip
It is worth noticing that by \cite[Proposition 3.1.6]{JL} the $Jf$-regularity of a mapping $f$ is weaker than the condition that each element of the set $\overline{co} (Jf(x)) \cup co(Jf(x)_{\infty} \backslash \{0\})$ is invertible. In the following example, we prove that it is, in fact, a strictly weaker condition. Thus, Theorem \ref{TFI} is more general that the one given in \cite{JL1}.

\begin{ex}\label{example}
{\rm Consider the mapping $f: \R^2\to\R^2$  defined by
$$
f(x,y) = (x-y, \, x+3y^{1/3}).
$$

It is not difficult to check that $f$ has the following pseudo-Jacobian:

\begin{align*}
Jf(x,y) = & \left\{
												 \begin{pmatrix}
												 1 & -1\\
												 1 & y^{- 2/3}
												 \end{pmatrix}
														\right\}\quad \text{for}\,\, y\neq 0\text{ , \, \, and} \quad
Jf(x,0) = & \left\{
												 \begin{pmatrix}
												 1 & -1\\
												 1 & \beta
												 \end{pmatrix}:\, \beta \geq 0
														\right\}.											
\end{align*}

It is clear that $Jf$ is usc. On the other hand, for every $(x, y)\in \mathbb R^2$, every matrix of  $Jf(x,y)$ is invertible. Moreover, for each $\beta\ge 0$ we have that
\begin{align*}
A_\beta^{-1} = \begin{pmatrix}
												 1 & -1\\
												 1 & \beta
												 \end{pmatrix}^{-1}= & \frac{1}{\beta + 1}\begin{pmatrix}
												 \beta & 1\\
												 -1 & 1
												 \end{pmatrix}
												 \end{align*}
so that												
\begin{align*}											
\|A_\beta^{-1}\|= & \sup\limits_{\|(u,v)\| \leq 1}\left \|\frac{1}{\beta +1}
												\begin{pmatrix}
												 \beta & 1\\
												 -1 & 1
												 \end{pmatrix}\begin{pmatrix}
												  u\\
												  v
												 \end{pmatrix}\right\|	& \leq  \sqrt{2} \sup\limits_{\|(u,v)\|_{\infty} \leq 1}\left \|\frac{1}{\beta +1}
												\begin{pmatrix}
												 \beta & 1\\
												 -1 & 1
												 \end{pmatrix}\begin{pmatrix}
												  u\\
												  v
												 \end{pmatrix}\right\|_{\infty}
												 \end{align*}
$$
\leq  \sqrt{2} \max \left\{1, \frac{2}{\beta+1}\right\} \leq 2 \sqrt{2}.
$$
Thus,  $/\hspace{-0.1cm}/A_\beta/\hspace{-0.1cm}/ = ||A_\beta^{-1}||^{-1} \ge \frac{1}{2 \sqrt{2}}$ for all $\beta\ge 0$. Therefore,  $\alpha_{Jf}(x,y)\geq \frac{1}{2 \sqrt{2}}$ for every $(x, y)\in \mathbb R^2$, so
$f$ is $Jf$-regular. Thus, as a consequence of Corollary \ref{bound} we obtain that $f: \R^2\to\R^2$ is a global homeomorphism. Nevertheless,  the recession cone of the set $co(Jf(0,0))$ is given by
\begin{align*}
(Jf(0,0))_{\infty} = & \left\{
												 \begin{pmatrix}
												 0 & 0\\
												 0 & \beta
												 \end{pmatrix}:\, \beta \geq 0
														\right\},													
\end{align*}
whose matrices are  not invertible.}
\end{ex}

\smallskip

In the last part of the paper we are going to obtain variants of the Hadamard integral condition in terms of the lower scalar Dini derivative and of the regularity index. First, we  need to recall the following concept from \cite{John}. Suppose that  $E$ and $F$ are Banach spaces, $D\subset E$ is an open set and  $f:D\to F$ is a local homeomorphism. Then for each $x\in D$ there is a neighborhood $S_x$ of $f(x)$ in $F$ such that $f$ has a local  inverse $f_x^{-1}$ in $S_x$. Moreover, as it was proved in \cite{John},  $S_x$ can be chosen as to be the so-called \emph{maximal star} with vertex $f(x)$, which is defined as the set of all points $z\in F$ for which the line segment $[f(x), z]$ can be lifted to a path $\gamma$ in $D$ starting at $x$, and such that $f$ maps homeomorphically the image $Im (\gamma)$ onto the segment $[f(x), z]$. The following properties  are also obtained in \cite{John}:
\begin{itemize}
\item[(i)] (Star-shaped) $S_x$ is an open neighborhood of $f(x)$, which is star-shaped with vertex  $f(x)$, that is, for each $z\in S_x$ the whole segment $[{f(x),z}]$ is also contained in $S_x$.
\item[(ii)] (Maximality) $S_x$ is maximal in the sense that for every sequence $(z_n) \subset S_x$ that lies on the same ray from $f(x)$ and converges to a point $z\not\in S_x$,  the sequence $\{f_x^{-1}(z_n)\}$ does not converge in $D$.
    \item[(iii)](Monodromy) For each path $q$ contained in $D$ connecting $x$ with some point $y$, and such that $f(Im (q))$ is contained in $S_x$, we have that $f_x^{-1} (f(y))=y$.
\end{itemize}

Our next result is a nonsmooth version of \cite[Theorem 2.1]{HoJinTi}, and provides a slight improvement of the condition in  \cite[Theorem IIA]{John}.

\begin{thm}\label{main:theorem}
Let $E$ and $F$ be Banach spaces, let  $\B(x_0,\rho)$ be an open ball of $E$ and let $f: \B(x_0, \rho) \to F$ be a local homeomorphism. Suppose that
\begin{equation}\label{condicion1}
\inf_{\|x-x_0\| \leq r}(D_{x}^{-} f) > 0 \quad \text{for} \quad0\leq r<\rho,
\end{equation}
and there exists a Riemann-integrable function  $\eta:[0,\rho) \to (0,\infty)$ such that
\begin{equation}\label{condicion2}
0< \eta(t)\leq \inf_{\|x-x_0\|=t}D_{x}^{-}f \quad \text{for} \quad 0\leq t < \rho.
\end{equation}
Then the maximal star $S_{x_{0}}$ contains the open ball $\B(f(x_0), \sigma)$, where
\begin{equation}
\sigma = \int_{0}^{\rho} \eta(t)dt.
\end{equation}
\end{thm}

\begin{proof} By considering the mapping $g(x)= f(x+x_0)-f(x_0)$ we may assume, without loss of generality, that
$x_0= 0$ and $f(x_0)=0$. Suppose that the maximal star $S_0$ does not contain the ball $\B(0, \sigma)$. Then there exists a vector $w\in \B(0, \sigma)$ such that $w \notin S_0$. Consider $u=\frac{w}{\Vert w \Vert}$, and define
$$
R:= \sup \{ \lambda \geq 0 \, : \, \lambda u \in S_0 \}.
$$
Then
$$R= \Vert R u \Vert \leq \Vert w \Vert < \sigma.$$
Now, let
$$
r:= \sup \{ \Vert f^{-1}(\lambda u)\Vert \, : \, 0\leq \lambda < R\}.
$$
By construction, it is clear that $r \leq \rho$. We are going to see that, in fact,  $r= \rho$. Indeed, if $r<\rho$ we have that $m=\inf_{\|x\|\leq r}(D_{x}^{-}f)>0$  by condition (\ref{condicion1}). Now,  fixed $\lambda_1,\lambda_2 \in [0, R)$ with $\lambda_1 <\lambda_2$ consider the path $p(t) = t u$ defined for $t\in [\lambda_1,\lambda_2]$ and set $q= f^{-1}\circ p$. Applying  Proposition \ref{MVI}, we obtain that, for some $\tau \in [\lambda_1,\lambda_2]$,
$$
|\lambda_2-\lambda_1|\geq (D_{f^{-1} \, (\tau u)}^{-}f)\|f^{-1}(\lambda_1 u) - f^{-1}(\lambda_2u)\| \geq m \, \|f^{-1}(\lambda_1 u) - f^{-1}(\lambda_2u)\|.
$$
This implies that
$$
\lim_{\lambda\to R}f^{-1}(\lambda u)
$$
exists in the closed ball ${\overline \B}(x_0, r)$, and this contradicts the maximality of $S_0$. Hence, $r= \rho$.

Now, since the map  $\lambda \mapsto \|f^{-1}(\lambda u)\|$ is continuous, it assumes all values between $0$ and $\rho$. Then,  for any sequence of values $0=t_0 < t_1 < \cdots < t_n < \rho$ there exist  $\lambda_1,\cdots, \lambda_n \in [0, R)$ such that $t_i=\|f^{-1}(\lambda_{i}u)\|$ for $i=1,\cdots,n$. For each fixed  $i\in \{1,\cdots,n-1\}$, suppose  that  $\lambda_i<\lambda_{i+1}$,  and  applying \cite[Lemma 2.1]{HoJinTi} we can find  $\lambda_{i}^{'} , \lambda_{i+1}^{'}\in [\lambda_i, \lambda_{i+1}]$ with $\lambda_i' <  \lambda_{i+1}'$ such that
\begin{equation*}
t_i = \|f^{-1}(\lambda_i u)\|=\|f^{-1}(\lambda_{i}^{'}u)\|\leq \|f^{-1}(\lambda u)\|
\leq \|f^{-1}(\lambda_{i+1}^{'} u)\|= \|f^{-1}(\lambda_{i+1}u)\| =t_{i+1}
\end{equation*}
for every $\lambda \in [\lambda_{i}^{'}, \lambda_{i+1}^{'}]$. Using again Proposition \ref{MVI}, there exists
$\tau_{i} \in [\lambda_{i}^{'},\lambda_{i+1}^{'}]$ such that
$$
(\lambda_{i+1} - \lambda_{i})\geq (\lambda_{i+1}^{'} - \lambda_{i}^{'})\geq(D_{f^{-1}(\tau_{i}u)}^{-}f)\|f^{-1}(\lambda_{i+1}^{'}u) - f^{-1}(\lambda_{i}^{'} u)\|.
$$
Hence,
\begin{align*}
t_{i+1}-t_i  &=
 \|f^{-1}(\lambda_{i+1}^{'}u)\|-\|f^{-1}(\lambda_{i}^{'}u)\|\\
 &\leq \|f^{-1}(\lambda_{i+1}^{'}u) - f^{-1}(\lambda_{i}^{'} u)\| \\
& \leq (D_{f^{-1}(\tau_{i}u)}^{-}f)^{-1}(\lambda_{i+1} - \lambda_{i}).
\end{align*}
Since  $\|f^{-1}(\tau_{i}u)\|  \in [t_i, t_{i+1}]$ for each $i\in \{1,\cdots,n-1\}$, we deduce from the condition (\ref{condicion2}) that
\begin{align*}
\sum_{i=0}^{n-1}\eta(\|f^{-1}(\tau_{i}u)\| )(t_{i+1} -t_{i}) &
\leq \sum_{i=0}^{n-1}(D_{f^{-1}(\tau_{i}u)}^{-}f)  (t_{i+1}-t_i) \leq \sum_{i=0}^{n-1}(\lambda_{i+1}-\lambda_{i}) = \lambda_{n} < R.
\end{align*}
As a consequence, we have that
$$
\sigma= \int_{0}^{\rho}\eta(t)dt \le R,
$$
which is a contradiction
\end{proof}

As a consequence of  Theorem \ref{main:theorem} above, we obtain a sufficient condition for global inversion by means of an integral condition in terms of the lower scalar  Dini derivative.

\begin{cor}\label{corol}
Let $E$ and $F$ be Banach spaces and  $f: E\to F$  a local homeomorphism such that
\begin{equation}\label{condicion4}
\inf_{\|x\| \leq r}(D_{x}^{-} f)>0 \quad \text{for} \quad0\leq r<\infty.
\end{equation}
Suppose that there exists a Riemann-integrable function  $\eta:[0,\infty) \to (0,\infty)$ such that
\begin{equation}\label{condicion5}
\int_{0}^{\infty}\eta(t)dt = \infty \qquad \text{and} \qquad 0<\eta(t)\leq \inf_{\|x\|=t}D_{x}^{-}f \quad \text{for} \quad 0\leq t < \infty.
\end{equation}
Then $f$ is a global homeomorphism from $E$ onto $F$. Moreover, for each $x\in E$,
\begin{equation}\label{desigualdad}
\|f(x)-f(0)\|\geq \int^{\|x\|}_{0}\eta(t)dt.
\end{equation}
\end{cor}
\begin{proof}
By (\ref{condicion5}) and Theorem \ref{main:theorem} we have that  every open ball  centered at $f(0)$ in $F$ is contained in the maximal star $S_0$. Therefore $S_0$ is the whole space $F$, and  $f$ maps $E$ onto $F$. On the other hand, from the monodromy  property $(iii)$ of the maximal star we obtain that  $f$ is one to one. Thus, $f$ is a global homeomorphism between $E$ and $F$.

Now, we will prove (\ref{desigualdad}). Let $x\in E$ nonzero,  consider  $\rho= \Vert x \Vert>0$ and  $\sigma = \int_{0}^{\rho}\eta(t)dt$, and suppose that $\|f(x)-f(0)\| < \sigma$. Applying Theorem \ref{main:theorem} to the restriction of $f$ to the open ball $\B(0,\rho)$, we obtain that $f(x)$ belongs to the corresponding maximal star $S_0$, and then, by the monodromy  property $(iii)$, we have that $x=f^{-1}(f(x))$ belongs to the open ball $\B(0,\rho)$, which is a contradiction. This completes the proof.
\end{proof}

In the next result we give an estimate for the domain of invertibility of a continuous mapping $f: \R^n\to\R^n$ which is a local homeomorphism around a point $x$, in terms of the regularity index of $f$ associated to a pseudo-Jacobian map.

\begin{cor}\label{estimate}
Let $f: \R^n\to\R^n$ be a continuous mapping and let $Jf$ be a usc pseudo-Jacobian map of $f$.  Let $x_0\in E$ and  $\rho>0$,  and suppose that  there exists a
Riemann-integrable function  $\eta:[0,\rho) \to (0,\infty)$ such that
$$
0<\eta(t)\leq\inf_{\|x-x_0\|=t}\alpha_{Jf}(x) \quad \text{for} \quad 0\leq t < \rho.
$$
Then $f(x_0 + \rho \B_n) \supset f(x_0)+ \sigma \B_n$ and $f$ admits a local inverse defined in the open ball $f(x_0) + \sigma \B_n$, where
$$
\sigma = \int_{0}^{\rho} \eta(t)dt.
$$
\end{cor}

\begin{proof}
By  Theorem \ref{TFI} and Theorem \ref{DerEsc1}, for applying Theorem \ref{main:theorem}  it only  remains  to show that condition (\ref{condicion1}) holds. Fix $0<r<\rho$. For each $x$ in the closed ball $x_0 + r {\overline \B}_n$, by Lemma \ref{conbeta} there exists $\beta_x>0$ such that $\alpha_{Jf}(x,\beta_x)>0$. The desired result follows by the compactness of $x_0 + r {\overline \B}_n$ and using again Theorem \ref{DerEsc1}.
\end{proof}

Finally, we obtain a version of the Hadamard integral condition for  a continuous mapping,  in terms of the regularity index associated to a pseudo-Jacobian map.

\begin{cor}\label{Hada}
Let $f: \R^n\to\R^n$ be a continuous mapping and let $Jf$ be a usc pseudo-Jacobian map of $f$. Suppose that   there exists a Riemann-integrable function  $\eta:[0,\infty) \to (0,\infty)$ such that
$$
\int_{0}^{\infty}\eta(t)dt = \infty \qquad \text{and} \qquad 0<\eta(t)\leq \inf_{\|x\|=t} \alpha_{Jf}(x) \quad \text{for} \quad 0\leq t < \infty.
$$
Then $f$ is a global homeomorphism. Moreover, for each $x\in E$,
$$
\|f(x)-f(0)\|\geq \int^{\|x\|}_{0}\eta(t)dt.
$$
\end{cor}
\begin{proof}
The result can be derived from Corollary \ref{corol}   following the lines of  Corollary \ref{estimate} above.
\end{proof}


\end{document}